\documentclass{amsart}
\usepackage{amsmath, amsfonts, amssymb,amsthm,mathtools,bm,color,mathrsfs}
\usepackage{latexsym}
\usepackage{amscd}
\usepackage[utf8]{inputenc}
\usepackage[all]{xy}
\usepackage{enumerate}
\usepackage[english]{babel}
\usepackage{csquotes}
\usepackage{comment}
\usepackage{cite}

\newtheorem{thm}{Theorem}[section]
\newtheorem{prop}[thm]{Proposition}
\newtheorem{lem}[thm]{Lemma}
\newtheorem{cor}[thm]{Corollary}

\theoremstyle{definition}

\newtheorem*{rmk}{Remark}

\newcommand\ff{\mathbf f}
\newcommand\CC{\mathbb C}
\newcommand\RR{\mathbb R}

\newcommand\NN{\mathbb N}
\newcommand\ZZ{\mathbb Z}
\newcommand\PP{\mathbb P}

\newcommand\ord{\operatorname{ord}}
\newcommand\exc{\operatorname{exc}}

\title{Quotient Problem for Entire functions with Moving Targets}
\author{Ji Guo}
\address{Institute of Mathematics, Academia Sinica, 6F, Astronomy-Mathematics Building, No. 1, Sec. 4, Roosevelt Road, Taipei 10617, TAIWAN} \email{jiguo@gate.sinica.edu.tw}

\begin{document}

\begin{abstract}
    As an analogue of the Hadamard quotient problem in number theory, the quotient problem (in the sense of complex entire functions) for two sequences of entire functions has been solved in \cite{guo2019quotient}.  In this paper, we consider the generalization of this problem in which we allow the coefficients to be entire functions of small growth by modifying the second main theorem with moving targets to a truncated version. We also compare our result to a special case in exponential polynomials first studied by Ritt \cite{ritt1929zeros}.
\end{abstract}

\keywords{Hadamard quotient problems, entire functions, moving targets, exponential polynomials}

\thanks{2010\ {\it Mathematics Subject Classification.} Primary 30D30, 
Secondary 32H30,11J97.}
\baselineskip=16truept 
\maketitle \pagestyle{myheadings}

\section{Introduction}

A sequence of numbers $\{{\bf G}(n)\}_{n\in\mathbb N}\subset\CC$ is called a {\it linear recurrence} if ${\bf G}(n+k)=c_0{\bf G}(n)+\dots+c_{k-1}{\bf G}(n+k-1)$ for all $n\in\mathbb N$ and for some constants $c_0,\dots,c_{k-1}\in\CC$.  Equivalently, $\{{\bf G}(n)\}_{n\in\mathbb N}$ has the following expression: 
$$
{\bf G}(n)=\sum_{i=1}^mg_i(n)\alpha_i^n, \quad\text{for all }n\in\mathbb N,
$$
where $g_i\in\mathbb C[X]$ are nonzero polynomials and $\alpha_i\in\mathbb C^*$ are distinct.
The recurrence is called ``simple" when all the $g_i$ are constant.

Analogous to the results in number theory in which  the quotient of two linear recurrences was considered (refer to \cite{corvaja1998diophantine, corvaja2002finiteness, zannier2005diophantine} for an overview), we have established the result on the divisibility of two ``simple linear recurrences of complex functions", generalizing the result in \cite{guo2019asymptotic}. (For similar problems in the non-Archimedean case or the case in several complex variables, one can also refer to \cite{pasten2016gcd} or \cite{liu2017upper} for more discussions.) 
\begin{thm}[\cite{guo2019quotient}]\label{simple}
    Let $l,m\ge 1$ be two positive integers. Let $f_1,\dots,f_l$ and $g_1,\dots,g_m$ be nonconstant entire functions such that $\max_{i=1,\dots,l} T_{f_i}(r)\asymp \max_{j=1,\dots,m} T_{g_j}(r)$. Let 
    $$
    F(n)=a_0+a_1f_1^n+\cdots+a_lf_l^n\quad \text{and}\quad G(n)=b_0+b_1g_1^n+\cdots+b_mg_m^n,
    $$
    where $a_0\in\CC$ and $a_1,\dots,a_l,b_0,\dots,b_m\in\CC^*$.    
    \begin{enumerate}[(i)]
        \item  If  the ratio $F(n)/G(n)$ is an entire function for infinitely many $n\in \ZZ^+$, or 
        \item  $f_1,\dots,f_l$ and $g_1,\dots,g_m$ are all units, i.e. entire functions without zero, and if the ratio $F(1)/G(1)$ is an entire function, 
    \end{enumerate}
    then $f_1^{i_1} \cdots f_l^{i_l}g_1^{j_1}\dots g_m^{j_m}\in K_{\mathbf g}$ for some $(i_1,\dots,i_l,j_1,\dots,j_m)\ne(0,\dots,0) \in\mathbb Z^{l+m}$.
\end{thm}
Here, $T_f( r)$ denotes the Nevanlinna characteristic function (refer to Section \ref{secpre}.)
The notation  $T_f(r)\asymp T_g(r)$ means that there exist positive numbers $a,b$ such that  $a T_f(r)<T_g(r)<b T_f(r)$  for $r$ sufficiently large.

Our main purpose in this paper is to generalize Theorem \ref{simple} by substituting small growth functions for the constant coefficients. This not only gives us a generalization of the quotient problem  for recurrence sequences \cite{corvaja2002finiteness}, it also gives new approaches to the study of exponential polynomials started by Ritt \cite{ritt1929zeros}.  

Before stating our main result, we introduce the following notations.
For entire functions $g_1,\dots,g_m$, let ${\bf g}=[1:g_1:\hdots:g_m]$ be a holomorphic map from $\CC$ to  $\PP^{m-1}$. We say a meromorphic function $a$ is of {\it slow growth} with respect to ${\mathbf g}$ if $T_{a}(r)=o(T_{\mathbf g}(r))$.
Let $K_{\mathbf g}:=\{ a | a \text{ is a meromorphic function and }T_{a}(r)=o(T_{\mathbf g}(r))\}$.  By the basic properties of characteristic functions,  $K_{\mathbf g}$ forms a field.  Let $R_{\mathbf g}\subset K_{\mathbf g}$ be the subring consisting of all entire functions in $K_{\mathbf g}$.

\begin{thm}\label{maintheorem}
    Let $l,m$ be two positive integers. Let $f_1,\dots,f_l$ and $g_1,\dots,g_m$ be nonconstant entire functions such that $\max_{1\leq i\leq l} T_{f_i}(r)\asymp \max_{1\leq j\leq m} T_{g_j}(r)$, and let $a_0\in R_{\mathbf{g}}$ and $a_1,\dots,a_l,b_0,\dots,b_m\in R_{\mathbf{g}}\setminus\{0\}$. Denote 
    $$
        F(n)=a_0+a_1f_1^n+\cdots+a_lf_l^n\quad \text{and}\quad G(n)=b_0+b_1g_1^n+\cdots+b_mg_m^n.
    $$
    \begin{enumerate}[(i)]
        \item  If  the ratio $F(n)/G(n)$ is an entire function for infinitely many $n\in \ZZ^+$, or 
        \item  $f_1,\dots,f_l$ and $g_1,\dots,g_m$ are all units, i.e. entire functions without zero, and if the ratio $F(1)/G(1)$ is an entire function, 
    \end{enumerate}
    then $f_1^{i_1} \cdots f_l^{i_l}g_1^{j_1}\dots g_m^{j_m}\in K_{\mathbf g}$ for some $(i_1,\dots,i_l,j_1,\dots,j_m)\ne(0,\dots,0) \in\mathbb Z^{l+m}$.
\end{thm}
In particular, applying this theorem  to exponential polynomials (\cite{ritt1927factorization,ritt1929zeros}), we obtain the following corollary.

\begin{cor}\label{corollary}
Let $F$ and $G$ be two exponential polynomials written as
$$
F(z)=a_0+a_1e^{\lambda_1z}+\cdots+a_le^{\lambda_lz}\quad \text{and}\quad G(z)=b_0+b_1e^{\tau_1z}+\cdots+b_me^{\tau_mz},
$$
where $a_i,b_j$ are non-zero polynomials in $\CC[z]$ and $\lambda_i,\tau_j$ are in $\CC$.  If $F(z)/G(z)$ is an entire function, then
$\lambda_1,\dots,\lambda_l,\tau_1,\dots,\tau_m$ are linearly dependent over $\mathbb Q$.
\end{cor}
\begin{rmk}
Ritt \cite{ritt1927factorization}  showed that {\it if $G(z)$ divides $F(z)$ (with $a_0=b_0=1$) in the ring of exponential polynomials with constant coefficients, then $\tau_1,\dots,\tau_m$ is a $\mathbb Q$-linear combination of $\lambda_1,\dots,\lambda_l$}. Then Everest and van der Poorten \cite{everest1997factorisation} generalized this to polynomial coefficients. Although the corollary is weaker than Ritt's result, we propose a new approach to solve this problem. Moreover, Ritt \cite{ritt1929zeros}, Lax \cite{lax1948quotient}, Rahman \cite{rahman1960quotient}, and Shields \cite{shields1963quotients} successively continued to study the quotient of two exponential polynomials and finally found that if it is an entire function, then the quotient would also be an exponential polynomial divided by a polynomial, similar to the conclusion obtained by Corvaja and Zannier \cite{corvaja2002finiteness}. 
\end{rmk} 

We give an brief proof of Corollary \ref{corollary} under the assumption that Theorem \ref{maintheorem} holds. 
\begin{proof}[Proof of Corollary \ref{corollary}]

Assume that $F(z)/G(z)$ is an entire function. Then our theorem implies that $$
Q(z):=\exp((i_1\lambda_1+\cdots+i_l\lambda_l+j_1\tau_1+\cdots+j_m\tau_m)z)\in K_{\mathbf{g}}
$$
for some non-trivial integers $i_1,\dots,i_l,j_1,\dots,j_m$. 
Notice that 
$$
T_{Q(z)}(r)=|i_1\lambda_1+\cdots+i_l\lambda_l+j_1\tau_1+\cdots+j_m\tau_m|\cdot r+O(1)
$$
which is not in $K_\mathbf{g}$ unless $i_1\lambda_1+\cdots+i_l\lambda_l+j_1\tau_1+\cdots+j_m\tau_m=0$.
\end{proof}

The work of Corvaja and Zannier on linear recurrence sequences in \cite{corvaja2002finiteness} and Vojta's dictionary between diophantine geometry and Nevanlinna theory (\cite{vojta2011diophantine, ru2001nevanlinna}) inspired us to derive Theorem \ref{simple} \cite{guo2019quotient}. To consider the case in which the constants $a_i,b_j$ are replaced with small growth functions with respect to $\mathbf{g}$, we first need to add a ramification term to the second main theorem for moving targets \cite{ru1991second} and derive a moving target version of Borel's lemma and Green's theorem. We can then adapt the proof in \cite{guo2019quotient} to Theorem \ref{maintheorem} and also use the ramification term to produce a truncated version and to reach a contradiction.

\section{Preliminary}
\label{secpre}
Now let us recall some notations, definitions and some basic results in Nevanlinna theory. Refer to \cite{lang1987introduction} or \cite{ru2001nevanlinna} for details.

Let $f$ be a meromorphic function  and   $z\in \CC$ be a complex number. Denote $v_z(f):=\ord_z(f)$,
$$v_z^+ (f):=\max\{0,v_z(f)\}, \quad\text{and }\quad  v_z^- (f):=-\min\{0,v_z(f)\}.$$
Let $n_f(\infty,r)$ denote the number of poles of $f$ in $\{z:|z|\le r\}$, counting multiplicity. The  {\it counting function} of $f$ at $\infty$ is defined by
\begin{align*}
N_f(\infty,r)&:=\int_0^r\frac{n_f(\infty,t)-n_f(\infty,0)}t dt+n_f(\infty,0)\log r\\
&=\sum_{0<|z|\le r } v_z^- (f)\log |\frac{r}{z}|+v_0^- (f)\log r.
\end{align*}
Then the {\it counting function} $N_f(a,r)$ for $a\in\CC$ is defined as
$$
N_f(a,r):=N_{1/(f-a)}( \infty,r).
$$
The  {\it proximity function} $m_f(\infty,r)$ is defined by
$$
m_f(\infty,r):=\int_0^{2\pi}\log^+|f(re^{i\theta})|\frac{d\theta}{2\pi},
$$
where $\log^+x=\max\{0,\log x\}$ for  $x\ge 0$. For any $a\in \CC,$ the {\it proximity function} $m_f(a,r)$ is defined by
$$m_f(a,r):=
m_{1/(f-a)}(\infty,r).
$$
The {\it characteristic function} is defined by
$$
T_f(r):=m_f(\infty,r)+N_f(\infty,r).
$$
It satisfies the inequalities $T_{fg}(r)\leq T_f(r)+T_g(r)+O(1)$ and $T_{f+g}(r)\leq T_f(r)+T_g(r)+O(1)$ for any entire functions $f$ and $g$. It also satisfies the First Main Theorem as follows.
\begin{thm}\label{Cfirstmain}  
    Let $f$ be a non-constant  meromorphic function on $\CC$.  Then for every $a\in\CC$ and for any positive real number $r$, $$m_f(a,r)+N_f(a,r)=T_f(r)+O(1),$$ where $O(1)$ is independent of $r$.
\end{thm}

The above theorem can be deduced from the following version of Jensen's formula.  
\begin{thm}\label{Jensen}
Let $f$ be a meromorphic function on $\{z: |z|\le r\}$ which is not the zero function.  Then
\begin{align*}
\int_0^{2\pi}\log|f(re^{i\theta})|\frac{d\theta}{2\pi}
&=N_f(r,0)- N_f(r,\infty)+\log |c_f|,
\end{align*}
where $c_f$ is the leading coefficient of $f$ expanded as the Laurent series in $z$, i.e.,
$f=c_fz^m+\cdots $ with $c_f\ne 0$.
\end{thm}
 
For a  holomorphic map ${\mathbf f}: \CC \rightarrow \PP^n(\CC)$, we take a reduced form of ${\mathbf f}=[f_0:\hdots:f_n]$, i.e. $f_0,\dots, f_n$ are  entire functions on $\CC$ without common zero. The Nevanlinna-Cartan {\it characteristic function}  $T_{\mathbf f}(r)$ is defined by
$$T_{\mathbf f}(r) =   \int_0^{2\pi} \log\|{\mathbf f}(re^{i\theta})\|\frac{d\theta}{2\pi}+O(1),$$
where $\|{\mathbf f}(z)\| = \max\{|f_0(z)|,\dots ,|f_n(z)|\}$.
This definition is independent, up to an additive constant, of the choice of the reduced representation of ${\mathbf f}$. Generally, if ${\mathbf f}=[f_0:\hdots:f_n]$ is not a reduced form, we define the height of $f$ as
\begin{equation*}
    T_{\mathbf f}(r) = \int_0^{2\pi} \log\|{\mathbf f}(re^{i\theta})\|\frac{d\theta}{2\pi}-\max_{i}\sum_{|z|\le r}\ord_{z}f_i\log\left|\frac{r}{z}\right|+O(1)
\end{equation*}
From the definition of the characteristic function, we derive the following proposition.
\begin{prop}[{\cite[Theorem~A3.1.2]{ru2001nevanlinna}}]
    \label{uplow}
    Let ${\mathbf f}=[f_0:\hdots:f_n]:\CC\to \PP^n(\CC)$ be  holomorphic curve, where $f_0,\dots,f_n$ are entire functions without common zero.  Then  
    \begin{equation}
        T_{f_j/f_i}(r)+O(1)\leq T_{\mathbf f}(r)\leq \sum_{j=0}^nT_{f_j/f_0}(r)+O(1).
    \end{equation}
\end{prop}

Let $H$ be a hyperplane in $\PP^n(\CC)(n>0)$ and let $a_0X_0+\cdots+a_nX_n$ be a linear form defining it. Let $P=[x_0:\hdots:x_n]\in\PP^n(\CC)\setminus H$ be a point. The Weil function $\lambda_H:\PP^n(\CC)\setminus H\to \RR$ is defined as
\begin{align}
    \lambda_H(P)=-\log \frac{|a_0x_0+\cdots+a_nx_n|}{\max\{|x_0|,\dots ,|x_n|\}}.
\end{align}
This definition depends on $a_0,\dots,a_n$, but only up to an additive constant and it is independent of the choice of homogeneous coordinates for $P$. 
The {\it proximity function} of ${\mathbf f}$ with respect to $H$ is defined by
$$ 
m_{\mathbf f}(H,r) =  \int_0^{2\pi} \lambda_{H}( {\mathbf f}(re^{i\theta}) ) \frac{d\theta}{2\pi}.
$$
Let ${\bf n}_{\mathbf f}(H,r)$ (respectively, ${\bf n}_{\mathbf f}^{(Q)}(H,r)$) be the number of zeros of $a_0f_0+\cdots+a_nf_n$ in the disk $|z|\le r$, counting multiplicity (respectively, ignoring multiplicity bigger than $Q\in\NN$).  The integrated counting function with respect to $H$ is  defined by
$$
N_{\mathbf f}(H,r) = \int_0^r \dfrac{{\bf n}_{\mathbf f}(H,t) - {\bf n}_{\mathbf f}(H,0)}{t} dt + {\bf n}_{\mathbf f}(H,0)\log r, 
$$ and the $Q$-truncated counting function with respect to $H$ is defined by 
$$
N_{\mathbf f}^{(Q)}(H,r) = \int_0^r \dfrac{{\bf n}_{\mathbf f}^{(Q)}(H,t) - {\bf n}_{\mathbf f}^{(Q)}(H,0)}{t} dt + {\bf n}_{\mathbf f}^{(Q)}(H,0)\log r. 
$$

The First Main Theorem also holds for hyperplanes. 
\begin{thm}\label{FMT}
Let  ${\mathbf f}: \CC \rightarrow \PP^n(\CC)$  be a holomorphic map and let $H$ be a hyperplane in $ \PP^n(\CC)$.  If $f(\CC)\not\subset H$, then for $r>0$,
\begin{align*} 
   T_{\mathbf f}(r) =m_{\mathbf f}(H,r)+N_{\mathbf f}(H,r)+O(1),
\end{align*}
where $O(1)$ is bounded independently of $r$.
\end{thm}

The following general second main theorem with ramification term is due to Vojta.
\begin{thm}[{\cite[Theorem 1]{vojta1997cartan}}]
\label{gsmt}
    Let ${\mathbf f}:\mathbb{C}\to\mathbb{P}^n(\mathbb{C})$ be a holomorphic curve whose image is not contained in any proper subspaces and let $[f_0:\hdots:f_n]$ be a reduced form of ${\mathbf f}$. Let $H_1,\dots,H_q$ be arbitrary hyperplanes in $\mathbb{P}^n(\mathbb{C})$. Denote by $W({\mathbf f})$ the Wronskian of $f_0,\dots,f_n$. Then for any $\varepsilon>0$, we have
$$
\int_0^{2\pi} \max_K \sum_{k\in K}\lambda_{H_k}({\mathbf f}(re^{\sqrt{-1}\theta}))\frac{d\theta}{2\pi}+N_{W({\mathbf f})}(0,r)\leq_{\operatorname{exc}} (n+1+\varepsilon)T_{\mathbf f}(r),
$$  
where the maximum is taken over all subsets $K$ of $\{1,\dots, q\}$ such that $H_k$ ($k\in K$) are in general position and $\leq_{\operatorname{exc}}$ means the estimate holds except for $r$ in a set of finite Lebesgue measure.
\end{thm}

We also need the following inequality with truncated counting functions. 
\begin{lem}\rm{(}\cite[Lemma~A3.2.1]{ru2001nevanlinna}\rm{)}\label{counting}
    Let ${\mathbf f}=[f_0:\hdots:f_n]:\mathbb{C}\to\mathbb{P}^n(\mathbb{C})$ be a holomorphic curve whose image is not contained in any proper subspaces and  $ f_0,\dots,f_n $ are entire functions with no common zero.  Let $H_1,\dots,H_q$ be the hyperplanes in $\PP^n$  in general position. Then 
    \begin{equation}
        \sum_{j=1}^qN_{\mathbf f}(H_j,r)-N_{W(\bf f)}(0,r)\leq \sum_{j=1}^qN_{\mathbf f}^{(n)}(H_j,r).
    \end{equation}
\end{lem}

Finally, we recall the following generalized Borel's lemma in \cite{ru2004truncated}.

\begin{thm}[{\cite[Theorem 2.1]{ru2004truncated}}]\label{trunborel}
    Let $\mathbf{f}=[f_0:\hdots:f_n]:\CC\to\PP^n(\CC)$ be a holomorphic map with $f_0,\dots,f_n$ entire and no common zero. Assume that $f_{n+1}$ is a holomorphic function satisfying the equation $f_0+\dots+f_n+f_{n+1}=0$. If $\sum_{i\in I}f_i\ne 0$ for any proper subset $I\subset\{0,\dots,n+1\}$, then 
    \begin{equation*}
        T_{\mathbf{f}}(r)\leq_{\exc} \sum_{i=1}^{n+1} N_{f_j}^{(n)}(0,r)+O(\log^+T_{\ff}(r)).
    \end{equation*}
\end{thm}



\section{Nevanlinna theory with moving targets}
\subsection{A Second Main Theorem with Moving Targets}
We will reformulate the second main theorem with moving targets (\cite{ru1991second}) to suit our purpose.
Let ${\bf f}:=[f_0:\hdots:f_n]$ be a holomorphic map from $\CC$ to  $\PP^n$ where $f_0,f_1,\dots,f_n$ are  holomorphic functions without common zero.
For the entire functions $\gamma_0,\dots,\gamma_n$, we let
\begin{align}\label{linearform0}
L=\gamma_{0} X_0+\dots+\gamma_{n} X_n.
\end{align}    
Then it defines a (moving) hyperplane $H$ in $\PP^n(K)$, where the field $K$ contains $\gamma_j$ for all $0\leq j\leq n$.
We note that for each $z\in\CC$, $H(z)$ is the hyperplane determined by the linear form $L(z) =\gamma_{0}(z) X_0+\dots+\gamma_{n}(z) X_n$.
In our convention, a (moving) hyperplane $H$ in $\PP^n(K)$ is assumed to be associated with a linear form as in \eqref{linearform0}.

For $1\le j\le q$, let $\gamma_{j0}, \cdots,\gamma_{jn}$ be entire functions of small growth with respect to $\mathbf{f}$ and let $K_{\gamma}$ be the field generated by all $\gamma_{ji}$.  Let
\begin{align}\label{linearform}
L_j:=\gamma_{j0} X_0+\dots+\gamma_{jn} X_n.
\end{align} 
Then each $L_j$ defines a  hyperplane $H_j$ in $\PP^n(K_{\gamma})$.  
Moving hyperplanes $H_1,\cdots,H_q$ are said to be in general position if any choice of $n+1$  linear forms $L_{i_1},\cdots, L_{i_{n+1}}$ among $\{L_1,\cdots,L_q\}$ are linearly independent over $K_{\gamma}$, or equivalently if for any choice of $n+1$  linear forms $L_{i_1},\cdots, L_{i_{n+1}}$ among $\{L_1,\cdots,L_q\}$, there exists $z\in\CC$ such that $L_{i_1}(z),\cdots, L_{i_{n+1}}(z)$ are linearly independent over $\CC$.
For a moving hyperplane $H$ determined by the linear form $L=\gamma_{0} X_0+\dots+\gamma_{n} X_n$ with $\gamma_0,\dots,\gamma_n\in K_{\gamma}$,  the Weil function $\lambda_H$   is defined as
\begin{align}
    \lambda_{H(z)}(P)=-\log \frac{|\gamma_0(z)x_0+\cdots+\gamma_n(z)x_n|}{\max\{|x_0|,\dots ,|x_n|\}\max\{|\gamma_0(z)|,\dots ,|\gamma_n(z)|\}},
\end{align}
where $P=(x_0,\cdots,x_n)\in \PP^n(\CC)$ and $z\in\CC$. We note that this function is well-defined except in a set of zero Lebesgue measure and is independent of the choice of homogeneous coordinates for $P$. 
The {\it proximity function} of ${\mathbf f}$ with respect to $H$ is defined by
$$ 
m_{\mathbf f}(H,r) =  \int_0^{2\pi} \lambda_{H(re^{\sqrt{-1}\theta})}( {\mathbf f}(re^{\sqrt{-1}\theta}) ) \frac{d\theta}{2\pi}.
$$
The integral is also well-defined except in a set of zero Lebesgue measure . 
Let ${\bf n}_{\mathbf f}(H,r)$  be the number of zeros of $\gamma_0f_0+\cdots+\gamma_nf_n$ in the disk $|z|\le r$, counting multiplicity. The integrated counting function with respect to $H$ is  defined by
$$
N_{\mathbf f}(H,r) = \int_0^r \dfrac{{\bf n}_{\mathbf f}(H,t) - {\bf n}_{\mathbf f}(H,0)}{t} dt + {\bf n}_{\mathbf f}(H,0)\log r.
$$ 
Then the first main theorem for a moving hyperplane $H$ \cite{ru1991second} can be stated as
\begin{equation}\label{fmtmov}
    T_{\mathbf{f}}(r) =N_{\mathbf{f}}(H,r)+m_{\mathbf{f}}(H,r)+o (T_{\bf f}(r)).
\end{equation}

Let $t$ be a positive integer and $V(t)$ be the vector space generated over $\CC$ by
\begin{align}\label{Vt}
    \left\{ \prod_{j=1}^q\prod_{k=0}^n \gamma_{jk}^{n_{jk}} \middle|n_{jk}\ge 0,\,\sum_{j=1}^q\sum_{k=0}^n n_{jk}\le t \right\}.
\end{align}
Choose entire functions $h_1=1,h_2,\cdots,h_w$ to be a basis of $V(t+1)$ such that $h_1,h_2,\cdots,h_u$ ($u\le w$) form a basis of $V(t)$.
Moreover, we have (\cite{steinmetz1986verallgemeinerung} or \cite{wang1996effective})
\begin{align}\label{dimV}
\liminf_{t\to\infty}\dim V(t+1)/\dim V(t)=1. 
\end{align}
Now we state a general version of the main theorem with moving targets. 
\begin{thm}
\label{movingsmt}
    Let $\ff=[f_0:\hdots:f_n]:\mathbb{C}\to\mathbb{P}^n(\mathbb{C})$ be a holomorphic curve and $f_0,\dots,f_n$ be entire functions with no common zero.
    Let $H_j$ ($1\le j\le q$) be arbitrary (moving) hyperplanes in $\mathbb{P}^n(K_{\gamma})$ defined by linear forms $L_j$ as in (\ref{linearform}). The notations $u,w,h_1,\dots,h_w$ are given as above. Denote by $W$ the Wronskian of $\{h_mf_k\,|\, 1\le m\le w,\,  0\le k\le n\}$. Assume that  $f_0,\dots,f_n$ are linearly independent over $K_{\gamma}$.  
    \begin{enumerate}
    \item For any $\varepsilon>0$, we have the following inequality 
    $$\int_0^{2\pi} \max_J \sum_{j\in J}\lambda_{H_j(re^{\sqrt{-1}\theta})}(\ff(re^{\sqrt{-1}\theta}))\frac{d\theta}{2\pi}+\frac1{u}N_{W}(0,r)\leq_{\operatorname{exc}} (n+1+\varepsilon)T_{\ff}(r),$$  where the maximum is taken over all subsets $J$ of $\{1,\dots, q\}$ such that $H_j(re^{\sqrt{-1}\theta})$ $(j\in J)$ are in general position.
    \item \label{trun_count} If the moving hyperplanes $H_{j_1},\dots,H_{j_\ell}$  are in general position for almost all $z\in\CC$,  where $\{j_1,\dots,j_\ell\}$ is a subset of $\{1,\dots,q\}$, then there exists a positive integer $Q$ such that
    \begin{equation*}
        \sum_{t=1}^{\ell} N_{\mathbf{f}}(H_{j_t},r)-\frac{1}{u}N_W(0,r)\leq \sum_{t=1}^{\ell} N_{\mathbf{f}}^{(Q)}(H_{j_t},r)+o(T_{\mathbf{f}}(r)).
    \end{equation*}

    \end{enumerate}

\end{thm}
\begin{proof}
    By \eqref{fmtmov}, the first main theorem,  we may assume that $q\geq n+1$ and that at least $n+1$ hyperplanes in $\{H_1,\dots,H_q\}$ are in general position. Define the holomorphic map as
    \begin{equation}\label{Func}
        \mathbf{F}:=[h_1f_0:h_2f_0:\hdots:h_w f_0:h_1f_1:\hdots:h_w f_n]:\CC\to\PP^{w(n+1)-1}(\CC).
    \end{equation}
    We note that this is a reduced form, i.e. $h_mf_k$, $1\le m\le w,\,  0\le k\le n$, are entire functions without common zero, since $h_1=1$ and $f_k$, $0\le k\le n$, have no common zero.
    Moreover, $\mathbf{F}$ is linearly non-degenerate over $\CC$ as $\mathbf{f}$ is linearly non-degenerate over $K_{\gamma}$ and as its characteristic function is in the same scale as $\ff$ by the following estimate.
    \begin{equation}\label{char_F}
        \begin{aligned}
            T_{\mathbf{F}}(r)&=\int_0^{2\pi}\log\max_{\substack{1\le i\le w\\0\le j\le n}}|h_i(re^{\sqrt{-1}\theta})f_j(re^{\sqrt{-1}\theta}) |\dfrac{d\theta}{2\pi} \\
            &= \int_{0}^{2\pi}\left(\log \Vert\mathbf{f}(re^{\sqrt{-1}\theta})\Vert +\log\max_{1\leq i\leq w}|h_i(re^{\sqrt{-1}\theta})|\right)\dfrac{d\theta}{2\pi} \\
            &\leq T_{\mathbf{f}}(r) + \int_{0}^{2\pi}\sum_{i=1}^w\log^+|h_i(re^{\sqrt{-1}\theta})|\dfrac{d\theta}{2\pi}\\
            &=T_{\mathbf{f}}(r)+\sum_{i=1}^w m_{h_i}(\infty, r)\le T_{\mathbf{f}}(r)+\sum_{i=1}^w T_{h_i}(r)\quad\text{(by Theorem \ref{Cfirstmain})}\\
            &\le T_{\mathbf{f}}(r)+o(T_{\mathbf{f}}(r)).
        \end{aligned}
    \end{equation}

    Next, we will construct a set of (fixed) hyperplanes in order to apply Theorem \ref{gsmt}. 
    We first observe that  each $h_iL_j=\sum_{k=0}^n h_i\gamma_{jk}X_k$, $1\le i\le u,\,  1\le j\le q$, is a linear form with coefficients in $V(t+1)$.  Therefore, for each $1\le i\le u,\,  1\le j\le q$, there exist $c_{ijk\nu}\in\CC$ such that 
    \begin{align*}
        h_iL_j=\sum_{k=0}^n\sum_{\nu=1}^w  c_{ijk\nu} h_{\nu} X_k.
    \end{align*}        
    For $i=1,\dots,u$ and $j=1,\dots,q$, let $\hat H_{ij}$ be the hyperplanes in $\PP^{w(n+1)-1}(\CC)$
    defined by  the following linear forms over $\CC$:
    \begin{align}\label{Lij}
        \hat L_{ij}=\sum_{k=0}^n\sum_{\nu=1}^w  c_{ijk\nu} X_{k\nu}.
    \end{align}
    It follows from the construction that
    \begin{align}\label{formrelation}
    h_iL_j(x_0,\dots,x_n)=\hat L_{ij}(h_1x_0,\dots,h_1x_n,\dots,h_wx_0,\dots,h_wx_n).
    \end{align}  
    
    Now applying Theorem \ref{gsmt} for $\mathbf{F}$ with the hyperplanes $\hat H_{ij}$, $1\le i\le u,\,  1\le j\le q$, it yields 
    \begin{equation}\label{gsmt3}
        \int_0^{2\pi} \max_{I,J} \sum_{\substack{i\in I\\ j\in J}}\lambda_{\hat H_{ij}}({\mathbf F}(re^{\sqrt{-1}\theta}))\frac{d\theta}{2\pi}+N_{W({\mathbf F})}(0,r)\leq_{\operatorname{exc}} (w(n+1)+\frac{\varepsilon}{2})T_{\mathbf F}(r),
    \end{equation}
    where the maximum ranges over all subsets $I\subset \{1,\dots,u\}$ and $J\subset\{1,\dots,q\}$ such that $\hat L_{ij}$ are linearly dependent over $\CC$.

    We first observe the following relation of  Weil functions of  $\hat H_{ij}$ and $H_{j}(z)$ for $1\le i\le u$, $1\le j\le q$, and $z\in\CC$.
    \begin{equation}\label{Weilrelation}
        \begin{aligned}
            \lambda_{\hat H_{ij}}({\mathbf F}(z)) 
            &=_{\exc} -\log \lvert \hat L_{ij}({\mathbf F}(z)) \rvert+\log \max_{1\leq m\leq w, 0\leq k\leq n} \lvert h_m(z) f_k(z) \rvert \\
            &=_{\exc}-\log |h_i(z) L_j(\ff)(z)| +\log \max_{1\leq m\leq w}|h_m(z)|\Vert \ff(z))\Vert  \text{\quad( by \eqref{formrelation})}\\
            &\geq_{\exc} -\log |L_j(\ff)(z)|+\log \Vert \ff(z))\Vert \\
            &=_{\exc}\lambda_{H_{j}(z)}({\mathbf f}(z))-\log \max_{0\leq k\leq n}|\gamma_{jk}(z)|.
        \end{aligned}
    \end{equation}

    Next, let $J$ be a subset of $\{1,\dots,q\}$ such that $\{H_j(z)\}_{j\in J}$ are in general position for some $z$.  Then  $\{L_j\}_{j\in J}$ must be linearly independent over $K_{\gamma}$.   We claim the following:
 
    \noindent{\bf Claim.}
    If $J$ is a subset of $\{1,\dots,q\}$ such that $\{L_j\}_{j\in J}$ are linearly independent over $K_{\gamma}$,
    then the hyperplanes $  \hat H_{ij}$, $1\le  i\le u$, $j\in J$ are in general position.  

    If the assertion fails, then there exist $\alpha_{ij}\in\CC$, not all zero, such that 
    $$
    \sum_{j\in J} \sum_{i=1}^u \alpha_{ij}\hat L_{ij}=0.
    $$
    Evaluating $\hat L_{ij}$ at $(h_1X_0,\dots,h_1X_n,\dots,h_wX_0,\dots,h_wX_n)$, where $X_0,\dots,X_n$ are variables, 
    it follows from (\ref{formrelation}) that
    \begin{align}\label{linearrel}
        \sum_{j\in J} \sum_{i=1}^u\alpha_{ij} h_iL_j(X_0,\dots,X_n)=0.
    \end{align} 
    Since $\{L_j\}_{j\in J}$ are linearly independent over $K_\gamma$, $\sum_{i=1}^u \alpha_{ij}h_i=0$ for each $1\le i\le u$. Also as $h_1,\dots,h_u$ are linearly independent over $\CC$, we have $\alpha_{ij}=0$ for any $1\le i\le u$ and $j\in J$.
   
    By the claim, together with (\ref{Weilrelation}), we have
    \begin{equation}\label{lambdaineq1}
        \begin{aligned}
            u\int_0^{2\pi} &\max_J \sum_{j\in J}\lambda_{H_j(re^{\sqrt{-1}\theta})}(\ff(re^{\sqrt{-1}\theta}))\frac{d\theta}{2\pi} \\ 
            &\le_{\exc} \int_0^{2\pi}\sum_{i=1}^u \max_{J} \sum_{ j\in J}\lambda_{\hat H_{ij}}({\mathbf F}(re^{\sqrt{-1}\theta}))\frac{d\theta}{2\pi}+o(T_{\mathbf{f}}(r))\\
            &\le_{\exc} \int_0^{2\pi} \max_{I,J} \sum_{\substack{i\in I\\ j\in J}}\lambda_{\hat H_{ij}}({\mathbf F}(re^{\sqrt{-1}\theta}))\frac{d\theta}{2\pi}+o(T_{\mathbf{f}}(r)) 
        \end{aligned}
    \end{equation}
             
    By \eqref{dimV}, we may choose $e$ large enough such that $w/u\leq 1+\frac\varepsilon{2(n+1)}$. We can then complete the proof of the first part by using this inequality together with \eqref{lambdaineq1}, \eqref{gsmt3},  and \eqref{char_F}.

    For Part (2), since  $f_0,\dots,f_n$ are linearly independent over $K_{\gamma}$, it is clear that the holomorphic map $\mathbf{F}$ in \eqref{Func} is linearly non-degenerate over $\CC$.  We have also proved the claim in the previous theorem that if $L_{j_1},\dots,L_{j_\ell}$  are linearly independent over $K_{\gamma}$, then the hyperplanes $\hat H_{ij_t}$ (determined  by linear forms in \eqref{Lij}), $1\le  i\le u$, $1\le t\le \ell$ are in general position over $\CC$.  Then we can apply Lemma \ref{counting} to the map $\mathbf{F}$ with the hyperplanes $\hat H_{ij_t}$, $1\le  i\le u$, $1\le t\le \ell$, to get    
    \begin{equation}
        \sum_{i=1}^u\sum_{t=1}^{\ell} N_{\mathbf{F}}(\hat H_{ij_t},r)-N_W(0,r)\leq  \sum_{i=1}^u\sum_{t=1}^{\ell} N_{\mathbf{F}}^{(Q)}(\hat H_{ij_t},r),
    \end{equation}
    where $Q:=w(n+1)-1$.
    It follows from \eqref{formrelation} that
    \begin{equation*}
        N_{\mathbf{F}}(\hat H_{ij},r)=N_{h_iL_j(\mathbf{f})}(0,r)\geq N_{L_j(\mathbf{f})}(0,r),
    \end{equation*}
    since the $h_i$ are entire functions,  and
    \begin{equation*}
        N_{\mathbf{F}}^{(Q)}(\hat H_{ij},r)=N_{h_iL_j(\mathbf{f})}^{(Q)}(0,r)\leq N_{L_j(\mathbf{f})}^{(Q)}(0,r)+N_{h_i}^{(Q)}(0,r).    \end{equation*}
    Since $N_{L_j(\mathbf{f})}(0,r)=N_\mathbf{f}(H_j,r)$ and $N_{h_i}^{(Q)}(0,r)\le T_{h_i}(r)\le o(T_{\mathbf{f}}(r))$, the above inequalities give
    \begin{equation}\label{trunQ}
        u\sum_{t=1}^{\ell} N_{\mathbf{f}}(H_{j_t},r)-N_W(0,r)\leq u\sum_{t=1}^{\ell} N_{\mathbf{f}}^{(Q)}(H_{j_t},r)+o(T_{\mathbf{f}}(r)).
    \end{equation}
\end{proof}

\subsection{Borel's lemma and Green's theorem with moving targets}
Before starting the proof of our main theorem, it is essential to give a generalization of Borel's lemma (\cite{borel1897zeros}) and Green's theorem  (\cite{green1975some}). 
We recall that $K_{\mathbf{f}}$ is the collection of meromorphic functions such that $T_{u}(r)=o (T_{\bf f}(r))$  and $R_{\mathbf{f}}$ is the subring of $K_{\mathbf{f}}$ consisting of all entire functions in $K_{\mathbf{f}}$.
\begin{lem}
    \label{Borel1}
    Let   $f_0,\dots,f_n$ be non-zero units, i.e. entire functions without zero, and let ${\bf f}=[f_0:\hdots:f_n]$ be a holomorphic map from $\CC$ to  $\PP^n$.
    If there exist $0\ne \gamma_i\in R_{\bf f}$ $(0\le i\le n)$ such that
    \begin{equation}\label{summation1}
        \gamma_0f_0+\dots+\gamma_nf_n=0,
    \end{equation}
    then for each $f_i$, there exists $j\ne i$ such that   $f_i/f_j\in K_{\bf f}$.
\end{lem}
The proof can be adapted easily from the one of Lemma \ref{green1}, so we omit it here.

\begin{lem}
    \label{green1}
    Let   $f_0,\dots,f_n$ be non-zero entire functions without common zero and let ${\bf f}=[f_0:\hdots:f_n]$ be a holomorphic map from $\CC$ to  $\PP^n$.
    Assume that   for an integer $k\ge n^2$ the following holds:
    \begin{equation}\label{summation2}
        \gamma_0f_0^k+\dots+\gamma_nf_n^k=0,
    \end{equation}
    where $0\ne \gamma_i\in R_{\bf f}$, $0\le i\le n$.
Then for each $f_i$, there exists $j\ne i$ such that   $(f_i/f_j)^k\in K_{\bf f}$.
\end{lem}
\begin{proof}
For a given $f_i$, there exists a vanishing subsum  of \eqref{propersubsum2} consisting of the term $\gamma_if_i^k$ and without any   vanishing proper subsum.  By reindexing, we may assume that this vanishing subsum is
    \begin{equation}\label{propersubsum2}
        \gamma_0f_0^k+\cdots+\gamma_m f_m^k=0,
    \end{equation}
    and hence $0\le i\le m$.
    If $m=1$, then $f_1^k=\alpha f_0^k$ for some $\alpha\in K_{\mathbf{f}}$.  Therefore we assume that $m\ge 2$. Let $\beta$ be an entire function such that $\tilde f_0=f_0/\beta,\dots,  \tilde f_{m-1}=f_{m-1}/\beta$ have no common zero, and let 
    $$
    \tilde{\mathbf f}_k:=[f_0^k:f_1^k:\hdots:f_{m-1}^k]= [\tilde f_0^k:\tilde f_1^k:\hdots:\tilde f_{m-1}^k].
    $$
    Let $h$ be an entire function such that $\gamma_0\tilde f_0^k/h,\gamma_1\tilde f_1^k/h,\cdots,\gamma_{m-1}\tilde f_{m-1}^k/h$ are entire functions with no common zero, and let 
    $$
        \mathbf{F}_k:=[\frac{\gamma_0\tilde f_0^k}h:\frac{\gamma_1\tilde f_1^k}h:\cdots,\frac{\gamma_{m-1}\tilde f_{m-1}^k}h]
    $$ 
    be a holomorphic map from $\CC$ to $\PP^{m-1}(\CC)$.
    Observe that 
    \begin{equation*}
        \max_{0\leq i\leq m-1}\{\log |\gamma_i\tilde f_i^k/h|\}\leq \max_{0\leq i\leq m-1}\{\log |\gamma_i|\}+\max_{0\leq i\leq m-1}\{\log |\tilde f_i^k|\}-\log|h|.
    \end{equation*}
    Then by the definition of characteristic functions, we conclude that
    \begin{equation}
        T_{\mathbf{F}_k}\leq T_{\tilde{\mathbf f}_k}+o(T_{\mathbf{f}}(r)),
    \end{equation}
    and similarly, by writing $ \tilde{\mathbf f}_k=[\frac h{\gamma_0 } \frac{\gamma_0\tilde f_0^k}h: \cdots:\frac h{\gamma_{m-1} } \frac{\gamma_{m-1}\tilde f_{m-1}^k}h]$, we have
    \begin{equation}\label{Charf_k2}
        T_{\tilde{\mathbf f}_k}\leq T_{\mathbf{F}_k}+o(T_{\mathbf{f}}(r)),
    \end{equation}
    
    Theorem \ref{FMT} and the positivity of the proximity function imply
    \begin{equation}\label{positivity}
        N_{\tilde f_i^k}(0,r)\leq T_{\tilde{\mathbf f}_k}
    \end{equation}
    for $i=0,\dots,m$ by taking the hyperplanes $H_i:=\{x_0:=0\}$ (for $i=0,\dots,m-1$) and $H_m:=\{x_0+\cdots+x_{m-1}=0\}$.
    Applying Theorem \ref{trunborel} to the map $\mathbf{F}_k$ with equation \eqref{propersubsum2},  we have
    \begin{equation}\label{truncation2}
        \begin{aligned}
            T_{\mathbf{F}_k}(r)&\leq_{\exc}  \sum_{i=0}^{m} N_{\gamma_i \tilde f_i^k/h}^{(m-1)}(0,r) +O(\log^+T_{\mathbf{F}_k}(r))\\
            &\leq \sum_{i=0}^{m} N_{\gamma_i}(0,r)+(\frac{m-1}{k})\sum_{i=0}^{m} N_{\tilde f_i^k}(0,r) +o(T_{\mathbf{f}}(r))\\
            &\leq \sum_{i=0}^{m} T_{\gamma_i}(r)+ \frac{(m+1)(m-1)}{k} T_{\tilde{\mathbf f}_k}(r) +o(T_{\mathbf{f}}(r)) \text{\quad by\ \eqref{positivity}}\\
            &\leq \frac{(m+1)(m-1)}{k} T_{\tilde{\mathbf f}_k}(r)+ o(T_{\mathbf{f}}(r)).
        \end{aligned}
    \end{equation}
    Together with \eqref{Charf_k2} and Proposition \ref{uplow}, this yields
    \begin{equation}
        \left(1-\frac{(m+1)(m-1)}{k}\right) T_{f_i/f_j }(r)\le \left(1-\frac{(m+1)(m-1)}{k}\right)T_{\mathbf{F}_k}(r) \le o(T_{\mathbf{f}}(r)).
    \end{equation}
    If $k\ge n^2>(m+1)(m-1)$, then 
    \begin{equation*}
        T_{f_i/f_j}(r)\leq o(T_{\mathbf{f}}(r))
    \end{equation*}
    for any $1\leq i<j\leq m$. Hence $f_i/f_j\in K_{\mathbf{f}}$ for any $1\leq i<j\leq m$. 
\end{proof}

Now we have finalized our preparation and we can begin the proof of our main theorem. 

\section{Proof of Theorem \ref{maintheorem} }

\begin{proof}[Proof of Theorem \ref{maintheorem} (i)]
    Assume that $f_1,\dots,f_l,g_1,\dots,g_m$ are entire functions such that $f_1^{i_1}\cdots f_l^{i_l}g_1^{j_1}\cdots g_m^{j_m}\notin K_{\mathbf{g}}$ for any index set  $(i_1,\dots,i_l,$ $j_1, \dots,j_m) \in\ZZ^{l+m}\setminus\{(0,0,\dots,0)\}$. Suppose that $$q(n):=\frac{F(n)}{G(n)}=\frac{a_0+a_1f_1^n+\cdots+a_lf_l^n}{b_0+b_1g_1^n+\cdots+b_mg_m^n}$$ is an entire function for a positive integer $n$.
    Since $$\max_{1\leq i\leq l} T_{f_i}(r)\asymp \max_{1\leq j\leq m} T_{g_j}(r),$$ there exist two positive constants $a,b$ such that 
    \begin{equation*}
        a\max_{1\leq j\leq m} T_{g_j}(r) \geq  \max_{1\leq i\leq l} T_{f_i}(r) \geq  b\max_{1\leq j\leq m} T_{g_j}(r).
    \end{equation*}
    Observe that there exists a subset $S$ of $\RR^+$ of infinite Lebesgue measure such that $\max_{1\leq j\leq m} T_{g_j}(r)=T_{g_k}(r)$ for $r\in S$ and for some $k\in\{1,\dots,m\}$. By rearranging the indices, we may assume that $k=1$. Thus $$T_{f_i}(r)\leq \max_{1\leq i\leq l} T_{f_i}(r)\leq a\max_{1\leq j\leq m} T_{g_j}(r)=aT_{g_1}(r)$$ for $1\leq i\leq l$ and $r\in S$. Without loss of generality, we may assume that $a>1$. Then for $r\in S$,
    \begin{align}\label{conditiona}
        T_{f_i}(r)\le a T_{g_1}(r),\, 1\le i\le l, \text{ and }  T_{g_j}(r)\le a T_{g_1}(r), \, 1\le j\le m.   
    \end{align}    
     
    Fix two positive integers $s,t$ to be determined later. Let 
    $$
    G_1(n)=G(n)-b_1g_1^n.
    $$ 
    Then
    \begin{equation}
        \label{Gn}
        G_1(n)^sq(n)=F(n)\left( \sum_{k=0}^{s-1}\binom{s}{k}G(n)^{s-1-k}(-b_1g_1^{n})^{k} \right)+(-b_1g_1^{n})^sq(n).
    \end{equation}
    We will use the following notation throughout the proof.  Denote ${\mathbf c}:=(0,c_{2},\dots,c_{m})\in (\ZZ_{\geq 0})^{m}$ and ${\mathbf d}:=(d_{1},\dots,d_{m})\in (\ZZ_{\geq 0})^{m}$.  Let  $|{\mathbf c}|:=c_{2}+\cdots+c_{m}$ and $|{\mathbf d}|=d_{1}+\cdots+d_{m}$.  We use the graded lexicographic order to arrange the index sets ${\mathbf c}\in (\ZZ_{\geq 0})^{m-1}$ and ${\mathbf d}\in (\ZZ_{\geq 0})^{m}$, i.e. ${\mathbf c}_i\succ {\mathbf c}_j$ if and only if $|{\mathbf c}_i|>|{\mathbf c}_j|$ or $|{\mathbf c}_i|=|{\mathbf c}_j|$ and the left-most nonzero entry of ${\mathbf c}_i-{\mathbf c}_j$ is positive.
    Denote by
    \begin{align}\label{alphabeta}
        \mathbf{g}^{n{\mathbf c}}:=g_2^{nc_2}\cdots g_m^{nc_m}\quad\text{and }\quad  \mathbf{g}^{n\mathbf d}=g_1^{nd_1}\cdots g_m^{nd_m}.
    \end{align}   
    For  each  ${\mathbf c}_i$ with $|{\mathbf c}_i|\leq t$, we define  
    \begin{equation}\label{phic}
        \varphi_{\mathbf{c}_i}:=\left[G_1(n)^sq(n)-F(n)\left( \sum_{k=0}^{s-1}G(n)^{s-1-k}(-b_1g_1^n)^{k}\right) \right]\mathbf{g}^{{n\mathbf c}_i}.
    \end{equation}
    Note that the number of such $\varphi_{\mathbf{c}_i}$ is 
    \begin{equation*}
        M=\binom{m-1+t}{m-1}.
    \end{equation*}
    Moreover, every $\varphi_{\mathbf{c}_i}$ is a linear combination of $\mathbf{g}^{n{\mathbf c}}q(n)$ where $|{\mathbf c}|\leq t+s$ and of the forms $\mathbf{g}^{n\mathbf d}f_i^{n}$ with $|{\mathbf d}|\leq s+t$ and $e_0\le i\le l$ where $e_0=1$ if  $a_0=0$ and $e_0=0$ if  $a_0\ne 0$ (letting $f_0=1$ in this case). Then the number of such forms $\mathbf{g}^{{n\mathbf c}}q(n)$ is
    \begin{equation*}
        N_1:=\binom{m-1+t+s}{m-1},
    \end{equation*}
    and the number of  ${\mathbf d}$ appearing in the above expression is denoted by $N_2$. 
    Denote $N:=N_1+(l+1)N_2$.   
    Let
    \begin{align}\label{xi}
        x_i(n)&:=\mathbf{g}^{{n\mathbf c}_i}q(n),  \quad  1\le i\le N_1,\quad \text{and}\cr
    x_{N_1+iN_2+j}(n)&:=f_i^n\mathbf{g}^{n{\mathbf d}_j}, \quad   e_0\le i\le l ,\, 1\le j\le N_2.
    \end{align}  
    Then it follows immediately  from \eqref{Gn} and \eqref{phic} that 
    \begin{equation}
        \label{g1n}
        \varphi_{{\mathbf c}_i}=(-b_1g_1^{n})^sq(n)\mathbf{g}^{{n\mathbf c}_i}=(-b_1)^sx_i(n)g_1^{sn}\quad \text{for  }1\le i \le M.
    \end{equation}
    Let $\mathbf{x}$ be the holomorphic map defined by 
    \begin{equation}
        \begin{aligned}\label{mapf}
            \mathbf{x}=[x_1(n):x_2(n):\hdots,x_N(n)]:\CC\to\PP^{N-1}(\CC).
        \end{aligned}
    \end{equation}
    We observe that  $a_ib_0^{s-1}f_i^n$, $e_0\le i\le l$ appear in the expansion of \eqref{phic} with ${\mathbf c}=(0,\dots,0)$.
    Since $b_0\ne0$,  we have $f_i^n$ , $e_0\le i\le l$, for our choice of $x_j(n)$.  Since they have no common zero, 
    $\mathbf{x}=(x_1(n),x_2(n),\dots,x_N(n))$ is a reduced form. 
    To simplify notation, we assume that $a_0\ne 0$ from now on.  The arguments are the same if $a_0=0$.
 
    For any $u\in K_{\mathbf{g}}$, notice that $T_u(r)\leq o(T_{\mathbf{g}}(r))\leq o(T_{\mathbf{x}}(r))$. 
    We claim  that this map is linearly independent over $K_{\mathbf{g}}$ if $n$ is sufficiently large. If the claim does not hold for a large enough $n$, there exist entire functions $\gamma_1,\dots,\gamma_{N_1},\mu_{0,1}$, $\mu_{0,2},\dots, \mu_{l,N_2}$ with no common zero in $K_{\mathbf{g}}$ which are not all zero such that 
    \begin{equation}
        \sum_{i=1}^{N_1}\gamma_{i}\mathbf{g}^{n{\mathbf c}_i}q(n)+\sum_{j=0}^l\sum_{k=1}^{N_2}\mu_{j,k}\mathbf{g}^{n{\mathbf d}_k}f_j^n=0,
    \end{equation}
    and hence 
    \begin{equation}
    \label{eqgm}
        \begin{aligned}
            &\sum_{i=1}^{N_1}\gamma_{i}\mathbf{g}^{n{\mathbf c}_i}(a_0+a_1f_1^n+\cdots+a_lf_l^n)\\
            &\quad +\sum_{j=0}^l\sum_{k=1}^{N_2}\mu_{j,k}\mathbf{g}^{n{\mathbf d}_k}f_j^n(b_0+b_1g_1^n+\cdots+b_mg_m^n)=0.
        \end{aligned}
    \end{equation}
    If all of the $\mu_{j,k}$ are zeros, then 
    \begin{equation}
       0= \sum_{i=1}^{N_1}\gamma_{i}\mathbf{g}^{n{\mathbf c}_i}(a_0+a_1f_1^n+\cdots+a_lf_l^n)=\sum_{i=1}^{N_1}\sum_{j=0}^la_j\gamma_{i}\mathbf{g}^{n{\mathbf c}_i}f_j^n.
    \end{equation} 
    Since not all of the  $\gamma_{i}$ are zeros,  by Lemma \ref{green1}, when $n>(l+1)^2N_1^2$, there exist two distinct terms $ f_j^n\mathbf{g}^{n{\mathbf c}_i}$ and $f_{j'}^n\mathbf{g}^{n{\mathbf c}_{i'}}$ such that the quotient 
    $$
    \frac{f_j^n\mathbf{g}^{n{\mathbf c}_i}}{f_{j'}^n\mathbf{g}^{n{\mathbf c}_{i'}}}=f_j^n f_{j'}^{-n}\mathbf{g}^{n({\mathbf c_{i}}-{\mathbf c_{i'})}}\in K_{\mathbf{g}},
    $$ 
    which contradicts the hypothesis that $f_1^{j_1}\cdots f_l^{j_l}g_1^{k_1}\cdots g_m^{k_m}\notin K_{\mathbf{g}}$ for any non-trivial index set $(j_1,\dots,j_l,k_1,\dots,k_m)$ $\in\ZZ^{l+m}$.  Therefore, we may assume that not all of the $\mu_{j,k}$ are zeros and let ${\mathbf d}_{k_0}$ be the maximal element (with respect to the graded lexicographic order) among the set $\{{\mathbf d}_k|\mu_{j,k}\ne 0$ for some $0\le j\le l\}$. 
    Naturally,  $\mu_{j_0,k_0}\ne 0$ for some $0\le j_0\le l$. 
    Expanding \eqref{eqgm} and using Lemma \ref{green1}  for $n>(l+1)^2(N_1+N_2(m+1))^2$,  we can find $f_{j'}^n\mathbf{g}^{n{\mathbf d}_{k'}}g_{j''}^n$ with $(j',k',j'')\ne (j_0,k_0,1)$,   or $f_{k'}^n\mathbf{g}^{n{\mathbf c}_{i}}$    among the zero terms of the expansion of \eqref{eqgm} such that 
    \begin{equation}\label{quot}
        \frac{f_{j_0}^n\mathbf{g}^{n{\mathbf d}_{k_0}}g_1^n}{f_{j'}^n\mathbf{g}^{n{\mathbf d}_{k'}}g_{j''}^n}\in K_{\mathbf{g}}  \quad \text{or} \quad \frac{f_{j_0}^n\mathbf{g}^{n{\mathbf d}_{k_0}}g_1^n}{f_{k'}^n\mathbf{g}^{n{\mathbf c}_{i}}}\in K_{\mathbf{g}}.
    \end{equation}
    By the definition of $\mathbf{g}^{n{\mathbf c_i}}$ in \eqref{alphabeta}, it  is clear that the second relation leads to a contradiction to the assumption that   $f_1^{j_1}\cdots f_l^{j_l}g_1^{k_1}\cdots g_m^{k_m}\notin K_{\mathbf{g}}$ for any non-trivial index set $(j_1,\dots,j_l,k_1,\dots,k_m)$ $\in\ZZ^{l+m}$.  
    Since $(j',k',j'')\ne (j_0,k_0,1)$ and the graded lexicographic order associated with the index set of  $\mathbf{g}^{n{\mathbf d}_k}g_1^n$ is bigger than the one with $\mathbf{g}^{n{\mathbf d}_{k'}}g_{j''}^n$ unless $(j',k')=(j_0,k_0)$, we can conclude similarly that the first quotient is not in $K_{\mathbf{g}}$, a contradiction.

    Next, we will construct a set of hyperplanes in $\PP^{N-1}(K_{\mathbf{g}})$ in order to apply Theorem \ref{movingsmt}.
    Since $G_1(n)=b_0+b_2g_2^n+b_3g_3^n+\cdots+b_mg_m^n$ with $b_0\ne 0$, the graded lexicographic order imposed on the ${\mathbf c}$ and the choice of the $x_i(n)$ give the following expression of $\phi_{{\mathbf c}_i}$ for $1\leq i\leq M$,
    \begin{equation}\label{phici}
        \begin{aligned}
            \varphi_{{\mathbf c}_i}&=A_{i,i}x_i(n)+A_{i,i+1}x_{i+1}(n)+ \cdots+A_{i,N}x_N(n),
        \end{aligned}
    \end{equation}
    where $A_{i,j}\in K_{\mathbf{g}}$, $1\le i\le M$, $i\le j\le N$ and  $A_{i,i}=b_0^s$ for each $i=1,\dots,M$. 
    Let
    \begin{equation}\label{hyperplane1}
        H_i :=\{X_{i-1}=0\}, \quad 1\le i\le N,    
        \end{equation}
        be the coordinate hyperplanes of $\PP^{N-1}(K_{\mathbf{g}})$, and  
    \begin{equation}\label{hyperplane}
        \begin{aligned}
            H_{N+i}&:=\{L_{N+i}:=A_{i,i}X_{i-1}+A_{i,i+1}X_i+\cdots+A_{i,N}X_{N-1}=0\},  \quad 1\le i\le M,
        \end{aligned}
    \end{equation}
    be hyperplanes according to the expression \eqref{phici} of $\varphi_{{\mathbf c}_i}$. It is clear from \eqref{hyperplane} that the hyperplanes $H_{M+1},\dots,H_{N+M}$ in $\PP^{N-1}(K_{\mathbf{g}})$  are in general position, and moreover $H_{M+1}(z),\dots,H_{N+M}(z)$    are in general position  for all $z\in\CC$ which are not a zero of $b_0$.
 Moreover, it's clear from  \eqref{phici} and \eqref{hyperplane} that \eqref{g1n} gives
 \begin{align}\label{Lvalue}
 L_{N+i} (\mathbf{x})=(-b_1)^sx_i(n)g_1^{sn}.
 \end{align}
  Let $e$ be any arbitrary large integer and  $V(e)$ be the $\CC$ vector space spanned by the set
    \begin{equation*}
        \left\{\prod_{ j=1}^ M\prod_{k=1}^N A_{jk}^{n_{jk}}\middle|n_{jk}\geq 0,\sum_{ j=1}^M\sum_{k=1}^N n_{jk}\leq e\right\}
    \end{equation*}
    and let $u=\dim V(e)$ and $w=\dim V(e+1)$.  Let $1=h_1,\dots,h_u$ be a basis of $V(e)$ and $h_1,\dots,h_w$ be a basis of $V(e+1)$.  Let $W$ be the Wronskian of $\{h_j x_k(n) |1\leq j\leq w,\, 1\leq k\leq N-1\}$.
    Now we apply Theorem \ref{movingsmt} to the map $\mathbf{x}=(x_1(n),x_2(n),\dots,x_N(n))$ and the hyperplanes $H_1,\dots,H_{N+M}$.  Then we obtain
    \begin{equation}\label{gsmt2}
        \begin{aligned}
            \int_{0}^{2\pi}\max_\mathcal{J}\sum_{j\in \mathcal{J}}\lambda_{H_{j}(re^{\sqrt{-1}\theta})}(\mathbf{x}(re^{\sqrt{-1}\theta}))\dfrac{d\theta}{2\pi}+\frac1{u}N_{W}(0,r)\leq_{\exc} (N+\varepsilon) T_{\mathbf{x}}(r),
        \end{aligned}
    \end{equation}
    where $\mathcal{J}$ runs over the subsets of $\{1,\dots,N+M\}$ such that the hyperplanes $H_j(re^{\sqrt{-1}\theta})$ $(j\in \mathcal{J})$ are in general position.
    
    We now proceed to derive a lower bound for the left hand side of (\ref{gsmt2}).  For any meromorphic function $\xi$, denote 
    $|\xi|_{r,\theta}:=|\xi(re^{\sqrt{-1}\theta})|$ and $\|\mathbf{x}\|_{r,\theta}:=\max_{1\le k\le N}| x_k(n) (re^{\sqrt{-1}\theta})|$.
    We claim that the following inequality holds for all $r$ outside of a set $E\subset (0,+\infty)$ with finite Lebesgue measure.
    \begin{equation}\label{lambda}
        \begin{aligned}
            \max_\mathcal{J}\sum_{j\in \mathcal{J}}\lambda_{H_{j}(re^{\sqrt{-1}\theta})}(\mathbf{x}(re^{\sqrt{-1}\theta}))\geq&  N\log \Vert \mathbf{x}\Vert_{r,\theta}-\sum_{i=1}^N\log|x_i(n)|_{r,\theta}+Msn\log^+ | g_1^{-1}|_{r,\theta} \\
            &\quad+Ms(\log^-   |b_0|_{r,\theta}-\log^+ |b_1|_{r,\theta}).
        \end{aligned}
    \end{equation}
    Since the zero set of an entire function is discrete, we may only consider $r$ with $b_0(re^{\sqrt{-1}\theta})\ne 0$ for  any $\theta$.  For $\theta\in S_{r}^-:=\{\theta:|g_1^n|_{r,\theta}<1\}$,  we choose $\mathcal{J}$ to be the set consisting of hyperplanes $H_{M+1}(re^{\sqrt{-1}\theta}),\dots,H_{N+M}(re^{\sqrt{-1}\theta})$ (since $b_0(re^{\sqrt{-1}\theta})\ne 0$) which are in general position and make the following computation.
    \begin{equation}\label{splus}
        \begin{aligned}
            &\sum_{j=M+1}^{N+M}\lambda_{H_{j}(re^{\sqrt{-1}\theta})}(\mathbf{x}(re^{\sqrt{-1}\theta}))\\
            &=\sum_{i=M+1}^N \log \frac{\Vert \mathbf{x}\Vert_{r,\theta}}{|x_i(n)|_{r,\theta}}+\sum_{i=1}^{M} \log \frac{\Vert \mathbf{x}\Vert_{r,\theta}\max_{i\leq j\leq N}\vert A_{i,j}\vert_{r,\theta}}{|b_1^sx_i(n)g_1^{sn}|_{r,\theta}} \quad\text{(by }\eqref{Lvalue}) \\
            &= N\log \Vert \mathbf{x}\Vert_{r,\theta}-\sum_{i=1}^N\log|x_i(n)|_{r,\theta}-Msn\log | g_1|_{r,\theta} +\sum_{i=1}^M(\log  \max_{i\leq j\leq N}\vert A_{i,j}\vert_{r,\theta}-\log |b_1^s|_{r,\theta})\\
            &\geq N\log \Vert \mathbf{x}\Vert_{r,\theta}-\sum_{i=1}^N\log|x_i(n)|_{r,\theta}-Msn\log^- | g_1|_{r,\theta}+M(\log   |b_0^s|_{r,\theta}-\log |b_1^s|_{r,\theta}),
        \end{aligned}
    \end{equation}
    where the last inequality follows from the identity $ A_{i,i}=b_0^s$.
    For $\theta\in S_{r}^+:=\{\theta:|g_1^n|_{r,\theta}\ge 1\}$, we choose $\mathcal{J}$ to be the set consisting of  hyperplanes $H_{1},\dots,H_{N}$ which are the coordinate hyperplanes of $\PP^{N-1}$.  Then
    \begin{equation}\label{sminus}
        \begin{aligned}
            \sum_{i=1}^{N}\lambda_{H_{i}(re^{\sqrt{-1}\theta})}(\mathbf{x}(re^{\sqrt{-1}\theta}))&=\sum_{i=1}^N \log \frac{\Vert \mathbf{x}\Vert_{r,\theta}}{|x_j(n)|_{r,\theta}}\\
            &=N\log \Vert \mathbf{x}\Vert_{r,\theta}-\sum_{i=1}^N\log|x_i(n)|_{r,\theta}-Msn\log^- | g_1|_{r,\theta},
        \end{aligned}
    \end{equation}
    since $\log^- | g_1^{sn}|_{r,\theta}^{-1}=0$
    on $S_r^+$.
  The assertion \eqref{lambda} is now verified by  \eqref{splus} and \eqref{sminus}.
  
    Integrating \eqref{lambda} over $d\theta$ from $0$ to $2\pi$, we derive from Theorem \ref{Cfirstmain}, Theorem \ref{Jensen} and the definition of the proximity and characteristic functions that
    \begin{equation}\label{lowb}
        \begin{aligned}
            &\int_0^{2\pi}\max_\mathcal{J}\sum_{j\in \mathcal{J}}\lambda_{H_{j}}(\mathbf{x}(re^{\sqrt{-1}\theta}))\dfrac{d\theta}{2\pi}\\
            &\geq  N T_{\mathbf{x}}(r)-\sum_{j=1}^N N_{x_j(n)}(0,r)+ Msn\cdot m_{g_1}(0,r)-Ms\cdot m_{b_0}(0,r)-Ms\cdot m_{b_1}(\infty,r)\\
            &\geq N T_{\mathbf{x}}(r)-\sum_{j=1}^N N_{x_j(n)}(0,r)+ MsnT_{ g_1}(r)-Msn N_{ g_1}(0,r)-Ms(T_{b_0}(r)+T_{b_1}(r))-O(1)\\
            &= N T_{\mathbf{x}}(r)+M s n T_{g_1}(r)-\sum_{j=M+1}^{N+M} N_{\mathbf{x}}(H_j,r)-Ms(T_{b_0}(r)+m_{b_1}(r))-O(1),
        \end{aligned}
    \end{equation}
    where the last one is due to the following identifications.
    \begin{equation}
        \begin{aligned}
            N_{\mathbf{x}}(H_{j},r)&=N_{x_j(n)}(0,r), \quad M+1\le j\le N,\quad\text{and}  \\
     N_{\mathbf{x}}(H_{N+j},r)& = N_{b_1^s}(0,r)+N_{x_j(n)}(0,r)+snN_{ g_1}(0,r), \quad 
            \quad 1\le j\le M,
        \end{aligned}
    \end{equation}
  by \eqref{Lvalue} and that $b_1$, $g_1$ and the $x_j(n)$ are entire functions.

    Since $H_{M+1},\dots,H_{N+M}$ are in general position, Theorem \ref{movingsmt}, \eqref{xi} and \eqref{Lvalue} imply there exists an integer $Q$ (may take $Q=w N-1$) such that
    \begin{equation}\label{trun2}
        \begin{aligned}
             \sum_{j=M+1}^{N+M}&N_{\mathbf{x}}(H_j,r)-N_{W}(0,r) \le \sum_{j=M+1}^{N+M}N_{\mathbf{x}}^{(Q)}(H_j,r)+o(T_{\mathbf{x}}(r))\\
            &\leq N Q \left(\sum_{i=1}^l N_{f_i}(0,r)+\sum_{j=1}^m N_{g_j}(0,r)\right)+ N_1N_{q(n)}(0,r)  + M s N_{b_1}(0,r)+o(T_{\mathbf{x}}(r))\\
            &\leq  N Q \left(\sum_{i=1}^l T_{f_i}(r)+\sum_{j=1}^m T_{g_j}(r)\right)+ N_1N_{F(n)}(0,r) + M s T_{b_1}(r)+o(T_{\mathbf{x}}(r))\\
            &\leq  N Q a(l+m)T_{g_1}(r)+ N_1T_{F(n)}(r)+o(T_{\mathbf{x}}(r)).
        \end{aligned}
    \end{equation}
     
    We also note that
    \begin{equation}\label{Fn}
        \begin{aligned}
            T_{F(n)}(r)&\leq \sum_{i=1}^l T_{f_i^n}(r)+\sum_{j=0}^l T_{a_j}(r) \leq  a l n T_{g_1}(r)+o(T_{\mathbf{x}}(r)).
        \end{aligned}
    \end{equation}

    Combining  \eqref{gsmt2}, \eqref{lowb}, \eqref{trun2} and \eqref{Fn} for $r\in S\setminus E$ large enough, we have
    \begin{equation}\label{finalinequality}
        \begin{aligned}
            (M s n- N Q a(l+m)- N_1 a l n) T_{g_1}(r) \le_{\exc} 2\varepsilon T_{\mathbf{x}}(r).
        \end{aligned}
    \end{equation}
    By the property of characteristic function, we obtain
    \begin{equation}
        \begin{aligned}
            T_{x_i(n)}(r)&\leq a(s+t)T_{g_1^n}(r)+T_{q(n)}(r) \text{\quad for \ } 1\leq i\leq N_1;\\
            T_{x_i(n)}(r)&\leq a(s+t+1)T_{g_1^n}(r) \text{\quad for \ } N+1\leq i\leq N;\\
            T_{q(n)}&\leq T_{F(n)}(r)+T_{G(n)}(r)\leq a(l+m)T_{g_1^n}(r).
        \end{aligned}
    \end{equation}
    So by Proposition \ref{uplow}, for $r\in S$ we have
    \begin{equation}
        \begin{aligned}
            T_{{\mathbf x}}(r)&\leq \sum_{j=1}^NT_{x_j(n)}(r)\\
            &\leq N(s+t+1)aT_{g_1^n}(r)+N_1T_{q(n)}(r)+O(1)\\
            &\leq N(s+t+1)aT_{g_1^n}(r)+N_1a(l+m)T_{g_1^n}(r)+O(1)\\
            &=an(N(s+t+1)+N_1(l+m))T_{g_1}(r)+O(1).
        \end{aligned}
    \end{equation}
    Thus from \eqref{finalinequality}, we conclude 
    \begin{equation}\label{finalineq2}
        \begin{aligned}
            ( M s n-N_1 a l n -N Q a(l+m))T_{g_1}(r)\leq_{\exc} 3\varepsilon an(N(s+t+1)+N_1(l+m))T_{g_1}(r)+O(1).
        \end{aligned}
    \end{equation}
    
    Finally, the parameters $s$, $t$ and $\varepsilon$ will be selected now to derive a contradiction from the above inequality.    
    To begin with, we fix $s>al$. 
    Since 
    \begin{equation*}
        Ms=s\binom{m-1+t}{m-1}=\frac{s}{(m-1)!}t^{m-1}+o(t^{m-1})
    \end{equation*}
    and
    \begin{equation*}
        \quad N_1al=al\binom{m-1+t+s}{m-1}=\frac{al}{(m-1)!}t^{m-1}+o(t^{m-1})
    \end{equation*}
    can be regarded as polynomials of $t$ with degrees both $m-1$ and the leading coefficient of $Ms$ is larger than the one for $aN_1l$, when $t$ is a sufficiently large integer, $Ms>N_1al$. Then we can choose $\varepsilon$ satisfying
    \begin{equation*}
        0<\varepsilon<\frac{Ms-aN_1l}{3a(N(s+t+1)+N_1(l+m))}.
    \end{equation*}
    Consequently, since $g_1$ is nonconstant, $T_{g_1}(r)(>0)$ is unbounded and we may deduce from \eqref{finalineq2} that
    \begin{equation*}
        n< n_0:=\frac{N Q a(l+m)}{M s-N_1 a l-3\varepsilon a(N(s+t+1)+N_1(l+m))}.
    \end{equation*} 
    In conclusion, if
    \begin{equation*}
        f_1^{i_1}\cdots f_l^{i_l}g_1^{j_1}\cdots g_m^{j_m}\notin K_{\mathbf{g}}
    \end{equation*}
    for any non-trivial index set $(i_1,\dots,i_l,j_1,\dots,j_m)\in\ZZ^{l+m}$, then the ratio $F(n)/G(n)$ is not an entire function for $n>\max\{n_0,n_1\}$, where $n_1=(l+1)^2(N_1+N_2(m+1))^2$ is such that $\mathbf{x}$ is linearly non-degenerate for $n> n_1$.
\end{proof}

For the second part of Theorem \ref{maintheorem}, firstly, we just need to replace Lemma \ref{green1} with Lemma \ref{Borel1} to conclude that the expression $x(1)$ is not contained in any proper linear subspace. Secondly, the facts that $N_{\xi}(0,r)=0$ for any unit $\xi$ and that $N_W(0,r)\geq 0$ imply that the left side of \eqref{trun2} is not greater than zero. Consequently, \eqref{finalinequality} becomes 
\begin{equation}\label{part2}
    \begin{aligned}
        M s  T_{g_1}(r) \le_{\exc} 2\varepsilon an(N(s+t+1)+N_1(l+m))T_{g_1}(r)+O(1).
    \end{aligned}
\end{equation}
Since $T_{g_1}>0$ is unbounded, we can choose $$0<\varepsilon< \frac{Ms}{2an(N(s+t+1)+N_1(l+m))}$$ and obtain 
\begin{equation}
    \begin{aligned}
        M s  T_{g_1}(r) > 2\varepsilon an(N(s+t+1)+N_1(l+m))T_{g_1}(r)+O(1)
    \end{aligned}
\end{equation}
when $r$ is sufficiently large, which contradicts \eqref{part2}.

\noindent\textbf{Acknowledgments.}  This work is one part of my PhD thesis, so I want to express my gratitude to my advisor Julie Tzu-Yueh Wang for her careful guidance and suggestions.


\end{document}